\documentclass[11pt]{amsart}
\usepackage{mathrsfs}
\usepackage{amssymb}
\usepackage[T1]{fontenc}
\usepackage{amscd}
\usepackage{amsmath, amssymb}
\usepackage{amsthm}
\usepackage{amsfonts}
\usepackage[colorlinks,linkcolor=blue,citecolor=blue, pdfstartview=FitH]
{hyperref}
\usepackage{amscd}
\usepackage{easybmat}
\usepackage{mathrsfs}
\usepackage{amsfonts}
\usepackage{color}
\usepackage{pifont}
\usepackage{upgreek}
\usepackage{bm}
\usepackage{hyperref}
\usepackage{shorttoc}
\usepackage{amsmath,amstext,amsthm,a4,amssymb,amscd}
\usepackage[mathscr]{eucal}
\usepackage{mathrsfs}
\usepackage{epsf}

\def\ra{\rightarrow}

\numberwithin{equation}{section}

\theoremstyle{plain}
\newtheorem{thm}{Theorem}[section]

\newtheorem{cor}[thm]{Corollary}
\theoremstyle{definition}

\theoremstyle{definition}
\newtheorem{defn}[thm]{Definition}
\newcommand{\comment}[1]{}

\let \cal \mathcal
\newcommand\vol{\mathrm{vol}}
\newcommand\area{\mathrm{area}}

\newcommand\sys{\mathrm{sys}}

\DeclareMathOperator{\N}{N}

\DeclareMathOperator{\arccosh}{arccosh}

\usepackage{amssymb}
\usepackage{amsthm}
\usepackage{amscd}

\usepackage{graphicx}%
\usepackage{fancyhdr}

\def\Z{\mathbb{Z}}
\def\Q{\mathbb{Q}}

\def\H{\mathbb{H}}

\DeclareMathOperator{\G}{\textbf{G}}

\DeclareMathOperator{\R}{\mathbb R}

\usepackage{fancyhdr}
\begin{document}

\title{Systole on locally symmetric spaces}

\makeatletter

\makeatother
\author{InKang Kim}

\address{Inkang Kim: School of Mathematics, KIAS, Heogiro 85, Dongdaemun-gu Seoul, 02455, Republic of Korea}
\email{inkang@kias.re.kr}

\address{}
\email{}

\begin{abstract}
 Here we survey on the growth of systoles of arithmetic locally symmetric spaces under the congruence covering and give simple proofs for the  best possible constants of Gromov for several important classes of symmetric spaces.
 \end{abstract}
\footnotetext[3]{\sl{Research by the author is partially supported by Grant NRF-2017R1A2A2A05001002.}}
\maketitle
\tableofcontents

\section*{Introduction}
A systole $\sys_{1}(M)$ of a Riemannian manifold $M$ is the smallest length of closed geodesics in $M$. It captures certain aspect of  geometry of $M$. For example, Margulis conjectured that there exists a uniform lower bound of systoles for arithmetic locally symmetric manifolds of finite volume. This conjecture has been supported by several cases. For example, the systole grows under the covering map, which was demonstrated, not long ago by
Buser and Sarnak in \cite{BS}, where they constructed examples of congruence coverings of an arithmetic hyperbolic surface whose systole grows logarithmically with respect to the area:
$$\sys_{1}(S)\geq \frac{4}{3}\log (\area(S))-c,$$ where $c$ is a constant independent of $S$. In 1996, Gromov \cite[Sec. 3.C.6]{Gr} showed that for any regular congruence covering $M_I$ of a compact arithmetic locally symmetric space $M$, there exists a constant $C>0$ so that
$$ \sys_{1}(M_I)\geq C \log(\vol(M_I))- d,$$ where $d$ is independent of $M_I$. Nonetheless,  an explicit value for the constant $C$ is unkown.
The purpose of this short paper is to provide the best possible constant $C$ for some locally symmetric spaces. We call $C$ a Gromov constant.

Our method is extremely simple in the sense that the techniques only use length-trace inequality, norm of a prime ideal in algebraic integers of a totally real number field, and the size of a  group over finite field. 
Especially we will use the following table throughout the paper.

\begin{table}[ht]
\caption{Size of classical semisimple groups over $\mathbb F_q$}
\centering
\begin{tabular}{c c}
\hline\hline
type of $G$    &  cardinality of $G(\mathbb F_q)$  \\ [0.5ex]
\hline
${}^1 A_r (r\geq 1)$   &     $q^{r(r+1)/2}\Pi_{j=1}^r(q^{j+1}-1)$ \\  [0.5ex]
${}^2 A_r (r\geq 2)$   &      $q^{r(r+1)/2}\Pi_{j=1}^r(q^{j+1}-(-1)^{j=1})$ \\ [0.5ex]
$B_r$ or $C_r (r\geq 2)$ &    $q^{r^2}\Pi_{j=1}^r (q^{2j}-1)$  \\ [0.5ex]
${}^1 D_r (r\geq 4)$       &   $q^{r(r-1)}(q^r-1)\Pi_{j=1}^{r-1}(q^{2j}-1)$ \\[0.5ex]
${}^2 D_r (r\geq 4)$       &   $q^{r(r-1)}(q^r+1)\Pi_{j=1}^{r-1}(q^{2j}-1)$ \\ [1ex]
\hline
\end{tabular}
\label{table}
\end{table}

For a given arithmetic lattice $\Gamma=\G(\cal O_F)$ defined over totally real number field $F$, we will consider the
congruence coverings corresponding to prime ideals $I\subset \cal O_F$, i.e.
$$\Gamma(I)=\Gamma \cap ker(\G(\cal O_F)\ra \G(\cal O_F/I))$$ via suitable embedding of
$\G$ into some linear group.
Denote  the corresponding cover of $M$ by $M_I$. 
We will estimate $C$ in the inequality 
$$ \sys_1(M_I)\geq C \log (\vol(M_I)) -c,$$ where $c$ is independent of $I\subset \mathcal O_F$.
\begin{thm}$C$ depends on the type of $\G$.
\begin{enumerate}
\item For arithmetic lattices in $SO(1,n)$, $C=\frac{4}{n(n+1)}$.
\item For arithmetic lattices in $SU(n,1)$ of the first type, $C=\frac{4}{n(n+2)}$, of the other type, $C=\frac{2}{n(n+2)}$.
\item For arithmetic lattices in $SL(n+1,\R)$, $C=\frac{\sqrt 2}{n(n+2)}$.
\end{enumerate}
\end{thm}

In the last section, we deal with the uniform lower bound for systoles of nonuniform arithmetic complex hyperbolic lattices
using complex Salem number.
\section{Real rank one case}
\subsection{Real and Quaternionic hyperbolic space}
For arithmetic lattices of the first type in real hyperbolic case, Murillo \cite{M} proved that
$$\sys_{1}(T_{I})\geq\frac{8}{n(n+1)}\log(\vol(T_{I}))-d.$$ 
For quaternionic hyperbolic lattice,  Emery-Kim-Murillo \cite{Kim} proved
$$ \sys_{1}(M_{I})\geq\frac{4}{(n+1)(2n+3)}\log\big(\vol(M_{I})\big)-d.$$

We give a systole growth inequality for real hyperbolic cases which have not been covered yet.
Let $F$ be a totally real number field and $\cal O_F$ the ring of integers of $F$. Let $a_1,\cdots, a_n\in \cal O_F$ such that $a_i>0$ for all $i$, but $\sigma(a_j)<0$ for any Galois embedding $\sigma\neq id$.
Then $G=SO(x_0^2-a_1x_1^2-\cdots - a_n x^2_n,\R)=SO(1,n)$, and $G(\cal O_F)$ is an arithmetic lattice of $G$. Even though this construction is exhaustive when $n$ is even, for $n$ odd, we need another
construction using quaternion algebra.

Let $\H^{a,b}_F$ be a quaternion algebra over $F$. Set
$$\tau_r(x_0+x_1 i+ x_2 j+x_3 k)=x_0+x_1 i- x_2 j+ x_3 k.$$  For $A\in GL(m, \H^{a,b}_F)$ with $\tau_r(A^T)=A$, let
$$SU(A, \tau_r, \H^{a,b}_F)=\{g\in SL(m,\H^{a,b}_F)| \tau_r(g^T) Ag= A\}.$$

Let $a$ and $b$ are nonzero real numbers such that either $a$ or $b$ is positive and $x$ is an invertible
element of $\H^{a,b}_{\R}$ such that $x=p+qi+sk$. Define the norm
$$N(x)=x\bar x=p^2-aq^2+abs^2,$$  and
\begin{equation}
\epsilon_{a,b}(x)=\left\{
\begin{aligned}
&1 && \text{if}\  bN_{a,b}(x)>0, \\
&2 && \text{if}\   bN_{a,b}(x)<0, \text{and}\\  
 &  &&  \text{either}  \left\{ \begin{aligned} 
                                       &b<0 \ \text{and}\ p>0,\ \text{or}\\
                                        & b>0\ \text{and}\ (a+1)q+(a-1)s\sqrt b>0,       \\ \end{aligned}\right. \\
&0  && \text{otherwise.}   \\
\end{aligned} \right.
\end{equation}

Let $a,b\in F^\times$ such that for each Galois embedding $\sigma$ of $F$, either $\sigma(a)$ or $\sigma(b)$ is positive, and choose $a_1,\cdots,a_m\in \H^{a,b}_F$ such that
\begin{enumerate}
\item $\tau_r(a_l)=a_l$
\item $\sigma(a_l)$ is invertible for each $l$ and $\sigma$
\item $\sum_{l=1}^m \epsilon_{a,b}(a_l)=1$, and
\item $\sum_{l=1}^m \epsilon_{\sigma(a),\sigma(b)}(\sigma(a_l))\in \{0,2m\}$ for each $\sigma\neq id$,
\end{enumerate}
Then
$\G=SU(diag(a_1,\cdots,a_m),\tau_r, \H^{a,b}_F)^0=SO(1,2m-1)^0$ and $\Gamma=\G(\cal O)$ is an arithmetic
lattice of $\G$ where $\cal O$ is an order of $\H^{a,b}_F$. For details, see \cite{Dave}.
As a vector space over $F$, $(\H^{a,b}_F)^m$ is isomorphic to $F^{4m}$, and hence one can identify
$GL(m, \H^{a,b}_F)$ with $GL(4m, F)$. Under this identification, $\G(\cal O)$ sits inside $GL(4m, \cal O_F)$.

A hyperbolic isometry in $SO(1,n)$ can be conjugate to a form
$$\begin{pmatrix}
       SO(n-1)   &  0 & 0 \\
          0         & \cosh r  & \sinh r \\
          0         & \sinh r   & \cosh r \end{pmatrix}$$
whose translation length $\ell_A$ is $2r$. Let $\lambda_1,\cdots,\lambda_{n_1}$ be eigenvalues of $SO(n-1)$ part of $A$. Then
$$|tr A|=|\lambda_1+\cdots+ \lambda_{n-1}+2\cosh r|\leq (n-1)+ 2e^r,$$ hence
$$|tr A|\leq (n-1)+ 2 e^{\frac{\ell_A}{2}}.$$  We get
$$\ell_A\geq 2 \log \frac{|tr A|-(n-1)}{2}.$$

For any $A=(a_{ij})\in \Gamma(I)$ for a prime ideal $I\subset \cal O_F$, where
$$\Gamma(I)=\Gamma \cap ker ( M(4m,\cal O_F)\ra M(4m, \cal O_F/I)),$$
$a_{ii}=1+c_i$ with $c_i\in I$, and since $G^\sigma$ is compact for any Galois embedding $\sigma\neq id$ of $F$,
$$|tr A^\sigma|=|\sum \sigma(1+c_i)|=|4m+\sum \sigma(c_i)|\leq 4m.$$
Hence $|\sum \sigma(c_i)|\leq 8m$. This impies
$$N(I)\leq N(\sum c_i)=|\sum c_i|\Pi_{\sigma\neq id}|\sigma(\sum c_i)|\leq (8m)^{f-1}|\sum c_i|,$$ where $f$ is a degree of $F$ over $\Q$.
Then
$$|tr A|=|4m+\sum c_i|\geq |\sum c_i|-4m\geq \frac{N(I)}{(8m)^{f-1}}-4m, $$ which implies
$$\ell_A\geq 2\log\frac{|tr A|-(2m-1-1)}{2}\geq 2\log N(I)- c .$$
Then we get
$$ \sys_1(M_I)\geq 2 \log (\N(I)) -c,$$ where $c$ is independent of $I\subset \mathcal O_K$.
 But $\textbf{G}=SO(1, n=2m-1)$ is of type ${}^2D_m$, hence accroding to the Table \ref{table}
$$|\textbf{G}(\mathcal O_k/I)|\leq \N(I)^{{m(2m-1)}}.$$

From $\vol(M_I)=\vol(M)[\Gamma:\Gamma_I]$ and $$0\ra \Gamma(I)\ra \Gamma=\G(\cal O_F)\ra \G(\cal O_F/I)\ra 0,$$
$$[\Gamma:\Gamma_I]\leq |SO(1,n;\cal O_F/I)|\leq N(I)^{m(2m-1)},$$
  for an arithmetic real hyperbolic manifold of dimension $n=2m-1\geq 7$
$$\sys_1(M_I)\geq \frac{4}{n(n+1)}\log(\vol(M_I))- d.$$
Note that for real hyperbolic manifold of dimension $n=2m\geq 4$, it is of type $B_m$, and
$$|\textbf{G}(\mathcal O_k/I)|\leq \N(I)^{{m(2m+1)}}.$$
Hence 
$$\sys_1(M_I)\geq \frac{4}{n(n+1)}\log(\vol(M_I))- d.$$
In any case, the constant is half of the one given by Murillo.

\subsection{Complex hyperbolic space}
\subsubsection{Arithmetic lattices in $SU(n,1)$}
Let $\Gamma\subset SU(n,1), \ n>1$ be an arithmetic lattice of the {\bf first type} defined over a totally real number field $K$ with a totally complex quadratic extension field $l$ such that $\Gamma$ is commensurable to 
$SU(h)(\mathcal O_l)$ where $h$ is a non-degenerate admissible Hermitian form over $V=l^{n+1}$ of signature $(n,1)$.
By  a proper embedding of $SU(h)$ in $GL_N$ for large $N$, we can assume that an element in $SU(h)(\mathcal O_l)$
is a matrix in $GL_N(\mathcal O_K)$. It suffices to take $N=2n+2$ as we can see below.

Another way to describe an arithmetic lattice of the first type defined over $K$ is as follows. Let $h$ be a Hermitian form over $l$ with coefficient matrix $A=(a_{ij})$, i.e.,
$$h(z)=z^* A z$$ with $a_{ij}\in l$. Then
$$\G=U(h, l)=\{T\in GL(n+1,l)| T^*AT=A\}.$$
Using the restriction of scalars $R_{l|K}(\G)$, $\G(l)$ corresponds to $R_{l|K}(\G)(K)$, and we can  ensure that
$U(h,l)$ sits inside  $GL(2n+2, K)$. Under this correspondence, $SU(h)(\cal O_l)$ sits inside
$GL(2n+2, \cal O_K)$.
The form $h$ is admissible if it has signature $(n,1)$, and for any non-identity embedding $\sigma:l \rightarrow \mathbb R$, the Hermitian form
$h^\sigma$ has signature $n+1$.


An arithmetic lattice of the {\bf second type} is constructed as follows. Let $L$ be a cyclic extension of $l$ and choose
embeddings $\lambda_j:L\ra \mathbb C$ compatible with $\tau_i:l\ra \mathbb C$ extending $\sigma_i:K\ra\mathbb R$.
Identify $L\ra\mathbb C$ via $\lambda_1$. Given a cyclic $l$-algebra $A$, identify $A\otimes_l L=M(n+1, L)$. For each
$\lambda_i\neq\lambda_1$, one obtains a new algebra $A^{\lambda_i}$.  Now take a unitary division algebra
$(A,*)$ of degree $n+1$ over $l$, with a Hermitian element $h$ (i.e., $h^*=h$) of signature $(n,1)$. 
Let $x^{\star h}=\mu_h\circ x^*=h x^*h$ and
$$SU(h)=\{x\in(A\otimes_l \mathbb C)^\times: xx^{\star h}=1\}=SU(n,1).$$

If $h$ is admissible, i.e. $SU(h^{\tau_j})$ is compact for $j\neq 1$, then
$$SU(h, \mathcal O)=\{x\in\mathcal O^{\times}: xx^{\star h}=1\}$$ is a cocompact lattice in $SU(n,1)$ where $\mathcal O$ is an $\mathcal O_l$ order in $A$.
This arithmetic lattice is called of the second type.
Even in this case, under the embedding $A\ra A\otimes_l L=M(n+1, L)$, an element in $SU(h,\mathcal O)$ lands in
$M(n+1, \mathcal O_l)$, and finally in $M(2n+2, \mathcal O_K)$.

Finally we introduce the {\bf mixed type} arithmetic lattices.  Let $(A,*)$ be a unitary algebra over $l$ of degree $d$. The simple $l$-algebra $M(r, A)$ admits an involution of second kind given by $*$-transposition.
If $L$ is a splitting field for $A$, then
$$M(r,A)\otimes_l L=M(r, A\otimes_l L)=M(r, M(d,L))=M(rd=n+1,L).$$
If a Hermitian element $h\in M(r,A)$ has signature $(n,1)$ in the splitting $M(r,A)\otimes_l L$,
$SU(h, A\otimes_l \mathbb C)=SU(n,1)$ and if $SU(h^{\tau_j}, A\otimes_l \mathbb C)$ is compact for any $\tau_j\neq \tau_1$, then $SU(h, \cal O)$ is an arithmetic lattice in $SU(n,1)$ for  any $\cal O_l$-order $\cal O$. Under the embedding $A\ra A\otimes_l L=M(d, L)$, an element in $SU(h,\mathcal O)$ lands in
$M(n+1, \mathcal O_l)$, and finally in $M(2n+2, \mathcal O_K)$.
For arithmetic lattices in complex hyperbolic space, refer to \cite{Mc}.

\subsubsection{Systole growth}
%

If  $A\in SU(n,1)$ is hyperbolic, it is conjugate to a matrix of the form
\begin{eqnarray}\label{hyperbolic}
 \begin{pmatrix}
 U  & 0 & 0\\
0   & \lambda & 0\\
0   &   0          &\beta \end{pmatrix}\end{eqnarray} where $U\in U(n-1)$, $\lambda=\overline{\beta^{-1}}$ with $|\lambda|>1$ and the other eigenvalues $\lambda_1,\cdots,\lambda_{n-1}$ all lie on the unit circle. The translation length of $A$ is
$$\ell_A=2\log |\lambda|.$$
Since $$|tr A|=|\lambda_1+\cdots+\lambda_{n-1}+\lambda+\beta|\leq (n-1)+2|\lambda|,$$
$$\ell_A\geq 2\log \frac{|tr A|-(n-1)}{2}.$$

For any $A=(a_{ij})\in \Gamma(I)$ for any prime ideal $I\subset \cal O_K$, where
$$\Gamma(I)=\Gamma \cap ker ( M(2n+2,\cal O_K)\ra M(2n+2, \cal O_K/I)),$$
$a_{ii}=1+c_i$ with $c_i\in I$,
$$|tr A^\sigma|=|\sum \sigma(1+c_i)|=|2n+2+\sum \sigma(c_i)|\leq 2n+2.$$
Hence $|\sum \sigma(c_i)|\leq 4n+4$. This impies
$$N(I)\leq N(\sum c_i)=|\sum c_i|\Pi_{\sigma\neq id}|\sigma(\sum c_i)|\leq (4n+4)^{f-1}|\sum c_i|,$$ where $f$ is a degree of $K$ over $\Q$.
Then
$$|tr A|=|2n+2+\sum c_i|\geq |\sum c_i|-2n-2\geq \frac{N(I)}{(4n+4)^{f-1}}-2n-2,$$ hence
$$\ell_A\geq 2 \log \frac{|tr A|-(n-1)}{2}\geq 2\log[\frac{N(I)}{2(4n+4)^{f-1}}-\frac{3n+1}{2}]\geq 2\log N(I)- c. $$

Hence we obtain 
$$ \sys_1(M_I)\geq 2 \log (\N(I)) -c,$$ where $c$ is independent of $I\subset \mathcal O_K$.
 $\textbf{G}=SU(n,1)$ is of type ${}^2A_n$, hence
$$|\textbf{G}(\mathcal O_k/I)|\leq \N(I)^{{n(n+2)}}.$$
Finally we obtain for an arithmetic complex hyperbolic manifold
$$\sys_1(M_I)\geq \frac{2}{n(n+2)}\log(\vol(M_I))- d.$$

For an arithmetic lattice of the first type,  more precise estimates in \cite{Kim} goes through for this case, and one obtains
$$\sys_1(M_I)\geq \frac{4}{n(n+2)}\log(\vol(M_I))- d.$$
\section{Arithmetic lattice in $SL(n,\R)$}
For a hyperbolic isometry $A\in SL(n,\R)$, the translation length of $A$ satisfies
\begin{eqnarray}\label{length}
\ell_A\geq \sqrt{2}\arccosh{\frac{|tr A|}{n}}\geq \sqrt{2}\log \frac{|tr A|}{n}.
\end{eqnarray}
See \cite{Jeff} for a proof. Note here that the metric on $SL(n,\R)/SO(n)$ is normalized so that the natural totally geodesic embedding of $SL(2,\R)/SO(2)$ has the sectional curvature equal to $-1$.

Let $L$ be a totally real quadratic extension of a totally real number field $F$ of degree $f$ over $\Q$ such that
$$\Gamma=SU(B=(b_1,\cdots,b_m), \tau; \cal O_D)=\{g\in SL(m, \cal O_D)| (g^\tau)^TB g=B \}$$
$$\subset SL(n=md, \cal O_F)$$ is a uniform lattice, where
$D$ is a central simple division algebra of degree $d$ over $L$, $b_i\in D^\times$,  $\tau$ is an anti-involution of $D$, $\tau(b_i)=b_i$, and $\cal O_D$ is an order of $D$. Here $SU(B, \tau; \cal O_D)^\sigma$ is  compact for any nontrivial Galois embedding $\sigma$ of $F$.
Note that for any $A\in \Gamma$,  the eigenvalues of $A^\sigma$ have norm 1, hence $|tr A^\sigma|\leq n$. 

Suppose $N(I)=|\cal O_F/I|$ is large for a prime ideal $I\subset \cal O_F$. Then for any
$A=(a_{ij}) \in \Gamma(I)$ where $$\Gamma(I)=\Gamma\cap ker (SL(n, \cal O_F)\rightarrow SL(n, \cal O_F/I)),$$
$a_{ii}=1+c_i$ for $c_i\in I$ and $a_{ij}\in I$.  Hence
$$|tr A^\sigma|=|\sum \sigma(1+c_i)|=|n+\sum \sigma(c_i)|\leq n.$$

This implies that 
$$|\sum\sigma(c_i)|\leq 2n.$$ Then
$$N(I)\leq N(c_1+\cdots+c_n)=|c_1+\cdots+ c_n|\Pi_{\sigma\neq id}| \sigma(\sum c_i)|\leq (2n)^{f-1}|c_1+\cdots+ c_n|.$$ Finally we get
$$|tr A|=|n+c_1+\cdots+ c_n|\geq |c_1+\cdots+ c_n|- n\geq \frac{N(I)}{(2n)^{f-1}}- n.$$
 
%

Hence  we obtain for $A\in \Gamma(I)$
$$\ell_A\geq \sqrt{2}\log \frac{|tr A|}{n}\geq   \sqrt{2}\log (\frac{N(I)}{n(2n)^{f-1}}-1)  \geq \sqrt 2 \log N(I)+ c .$$

Since $SL(n+1, \R)$ is of type $A_n$, according to Table \ref{table}
$$|SL(n+1,\cal O_F/I)|\leq N(I)^{n(n+2)}.$$

Another way to construct an arithmetic lattice of $SL(n,\R)$ is as follows.
Let $D$ be a central division algebra of degree $d$ over $\Q$, such that $D$ splits over $\R$, and
$\cal O_D$ an order.  Then $\phi(SL(m,\cal O_D))$ is an arithmetic subgroup of $SL(n=dm,\R)$ for any
embedding $\phi:D\ra M(d,\R)$. Then the same argument as above with $F=\Q$ goes through.

Since these two cases cover all types of arithmetic lattices in $SL(n+1,\R)$, see \cite{Dave}, previous arguments give:
\begin{thm}The systol growth of  arithmetic lattices in $SL(n+1,\R)$ under the congruence cover is
$$\sys_1(M_I)\geq   \frac{\sqrt 2}{n(n+2)}   \log(\vol(M_I))-d.$$
\end{thm}
\begin{proof}It follows from $\vol(M_I)=\vol(M)[\Gamma:\Gamma_I]$ and
$[\Gamma:\Gamma_I]\leq |SL(n+1,\cal O_F/I)|\leq N(I)^{n(n+2)}.$
\end{proof}
Note that for $n=1$, $SL(2,\R)$ is the isometry group of real hyperbolic plane, and the constant is
$\frac{\sqrt 2}{3}$, which is less than $\frac{4}{3}$.
But Equation (\ref{length}) is responsible for $\sqrt 2$. Indeed in $SL(2,\R)$ case, we have a better
estimate
$$\ell_A\geq 2 \log \frac{|tr A|}{2},$$ and hence the constant is $\frac{2}{3}$ instead of $\frac{\sqrt 2}{3}$. To get the missing factor 2, we need a better estimate for the length-trace inequality to get $N(I)^2$ factor.

This suggests that our constant is not optimal. On the other hand, the constant provided  in \cite{Jeff}
is cubic in $n$.

\section{Uniform lower bound for systoles of nonuniform arithmetic manifolds in complex hyperbolic space}
In this section we prove that
\begin{thm}There exists a uniform lower bound for the length of closed geodesics in non-unform arithmetic
lattices  of $SU(n,1)$ for any $n<N$, once $N$ is fixed. 
\end{thm}
First of all, we need a concept of complex Salem number.
\begin{defn}A complex number $\lambda=r e^{i\theta} (r>1) $ is a {\bf complex Salem number} if  $\lambda$ is a root of monic polynomial over $\mathbb Z$ such that other roots except $\lambda,\bar\lambda$ and $\bar\lambda^{-1}, \lambda^{-1}$ all lie on the unit circle. The monic minimal polynomial of $\lambda$ over $\Q$ is called a Salem polynomial of $\lambda$, $s_\lambda(z)=s(z)$.
\end{defn}
Here are some examples of minimal polynomials of complex Salem numbers.
$$1 + x + x^2 - x^3 + x^4 + x^5 + x^6 $$
$$1 + x^3 + x^4 + x^5 + x^8 $$
$$1 + x - x^5 + x^9 + x^{10}.$$

\begin{thm}Let $\Gamma\subset SU(n,1)$ be a complex hyperbolic arithmetic lattice  over $K$. Then for any hyperbolic isometry $\gamma$, the largest (in absolute value) eigenvalue is a complex Salem number.
\end{thm}
\begin{proof}
Let $\gamma\in \Gamma$ be a hyperbolic isometry, hence $$\gamma^m\in SU(h,\mathcal O)\subset GL(N,\mathcal O_K)$$ for some $m$. Indeed, one can take $N=2n+2$.  If $p(z)$ is the characteristic polynomial of $\gamma$ and $p'(z)$ is the characteristic polynomial of $\gamma^m$, then $p'(z)$ is defined over $\mathcal O_K$ since matrix in $GL(N,\mathcal O_K)$ has coefficients in $\mathcal O_K$. The roots of $p'(z)$ are the $m$th powers of the roots  of $p(z)$. Since the coefficients of $p'(z)$ are in $\mathcal O_K$, the roots of $p'(z)$ are integral over $\mathcal O_K$. Since $\mathcal O_K$ is integral over $\Z$, the roots of $p'(z)$ are integral over $\Z$, i.e. algebraic integers. Since the set of algebraic integers in $\mathbb C$ are integrally closed,  the roots of $p(z)$ are algebraic integers. Hence its coefficients of $p(z)$ are in $\mathcal O_K$. 

 Furthermore
$\gamma$ is conjugate to a matrix of the form (\ref{hyperbolic}). Hence $p(z)$ has a root $\lambda=re^{i\theta}, r>1$, $\bar\lambda^{-1}$, and all other roots on the unit circle. Let $p(z)=p_1(z)\cdots p_k(z)$ be a factorization into monic irreducible polynomials over $\mathcal O_K$ such that $p_1(\bar\lambda^{-1})=0$. Then one can show that $\lambda,\bar\lambda^{-1}$ are roots of $p_1(z)$ as follows.

Suppose not. Then the constant term $w \bar\lambda^{-1}$ of $p_1(z)$ is in $\mathcal O_K$ where $w$ is some unit complex number. Choose non-identity $\sigma:K\rightarrow \mathbb R$ such that $\pm 1\neq\sigma(w\bar\lambda^{-1})\in\mathcal O_{\sigma K}$. Then the absolute value of $\sigma(w\bar\lambda^{-1})$ cannot be 1. But this contradicts that $\sigma(w\bar\lambda^{-1})$ is a root of
$p^\sigma(z)$ which is a characteristic polynomial of $\gamma^\sigma$ in a compact unitary group $SU(h^{\tau_j}, A\otimes_l \mathbb C)$ for some  $\tau_j:l\ra \mathbb C$.

Suppose $\tau_i:l \rightarrow \mathbb C, i=1,\cdots, d=[l:\Q]$ are embeddings of $l$ into $\mathbb C$ and $\tau_1=id$ and $\tau_2=\bar\tau_1$ is a complex conjugation, $\tau_{2k}=\bar \tau_{2k-1}$.  Then $\tau_{2k-1}|K:K\rightarrow \mathbb R$ are Galois embeddings of $K$ into $\mathbb R$. Note that $p_1(z)p_1^{\tau_2}(z)$ has $\lambda,\bar\lambda,\lambda^{-1},\bar\lambda^{-1}$ as roots. Set $p_1^*(z)=\Pi p_1^{\tau_i}(z)$. Since $p_1^*$ is invariant under $\tau_i$, $p_1^*(z)$ is over $\Z$. Since $p_1(z)$ is a factor of $p(z)$, each $p_1^{\tau_i}$ is a factor of $p^{\tau_i}(z)$ for $i>1$.
But since $h^{\tau_i}$ is unitary for $i=2k-1>1$, $p^{\tau_i}(z)$ must have roots all on the unit circle.
This shows that $p_1^*(z)$ has roots all on the unit circle except $\lambda,\lambda^{-1},\bar\lambda, \bar\lambda^{-1}$. If $s(z)$ is an irreducible minimal polynomial over $\Z$ of $\lambda$, $s(z)$ must divide $p_1^*(z)$. If $s(z)=p_1^*(z)$, then $\lambda$ is a complex Salem number and $s(z)$ is a Salem polynomial. Hence it suffices to show that $p_1^*(z)$ is irreducible over $\mathbb Z$. Suppose $p_1^*(z)=s(z)h(z)$ for some monic polynomial $h(z)$ over $\Z$. For a root $r$ of $h(z)$, it is a root of $p_1^\tau(z)$ for some $\tau=\tau_j$. Then $p_1^\tau(z)$ is a minimal polynomial of $r$ over $\tau(K)$ since the same thing is true for $p_1(z)$ and any root of it. Hence $p_1^\tau$ divides $h$ in $\tau(K)[z]$. Applying $\tau^{-1}$, since $h^{\tau^{-1}}=h$ ( $h$ is over $\Z$), we see that $p_1$ divides $h$ in $K[z]$. Hence $\lambda$ is a root of $h(z)$, but $\lambda$ is a simple root of $p_1^*(z)$. Hence $s(z)=p_1^*(z)$.
\end{proof}

Note that the degree of $\lambda$ over $\mathbb Q$ is $\leq 2(n+1) [l:\Q]$.
If $\Gamma$ is non-uniform, then necessarily $K=\Q$. 

Since there are only finitely many monic polynomials of bounded degree with bounded Mahler measure \cite{A}, the set of
complex Salem numbers of bounded degree over $\mathbb Q$ has a least element.
Consequently we have
\begin{cor}If $\Gamma\subset SU(n,1)$ is a non-uniform arithmetic lattice, then the length of
a closed geodesic is uniformly bounded below for any $n\leq N$ for a fixed $N$.
\end{cor}

For real hyperbolic case, the above corollary is proved in \cite{Vincent}.


\end{document}